\documentclass[11pt]{article}

\usepackage{tikz}
\usepackage{subfigure}
\usepackage[english]{babel}

\usepackage[center]{caption2}
\usepackage{amsfonts,amssymb,amsmath,latexsym,amsthm}
\usepackage{multirow}

\def\ex{\mbox{ex}}

\newtheorem{thm}{Theorem}[section]

\newtheorem{lem}[thm]{Lemma}
\newtheorem{prop}[thm]{Proposition}
\newtheorem{conj}[thm]{Conjecture}

\newtheorem{obs}[thm]{Observation}
\newtheorem{claim}[thm]{Claim}

\begin{document}

\title{Generalized Tur\'an problem with bounded  matching number
\thanks{The work was supported by National Natural Science Foundation of China (No.12401455) to Yue Ma;  National Natural Science Foundation of China (No.12471336, 12071453) and Innovation Program for Quantum Science and Technology (2021ZD0302902) to Xinmin Hou.}}
\author{Yue Ma$^a$, \quad Xinmin Hou$^{b,c}$, \quad Zhi Yin$^b$\\
\small $^{a}$School of Mathematics and Statistics,\\
\small Nanjing University of Science and Technology, Nanjing, Jiangsu 210094, China.\\
\small $^{b}$ School of Mathematical Sciences\\
\small University of Science and Technology of China, Hefei, Anhui 230026, China.\\
\small $^{c}$ Hefei National Laboratory,\\
\small University of Science and Technology of China, Hefei, Anhui 230088, China.\\
\small $^a$yma@njust.edu.cn; $^b$xmhou@ustc.edu.cn, yinzhi@mail.ustc.edu.cn
}

\date{}

\maketitle

\begin{abstract}
For a graph $T$ and a set of graphs $\mathcal{H}$, let $\ex(n,T,\mathcal{H})$ denote the maximum number of copies of $T$ in an $n$-vertex $\mathcal{H}$-free graph.
Recently, Alon and Frankl~( Journal of Combinatorial Theory, Series B, 2024) determined the exact value of $\ex(n,K_2,\{K_{k+1},M_{s+1}\})$, where $K_{k+1}$ and $M_{s+1}$ are complete graph on $k+1$ vertices and matching of size $s+1$, respectively. 
In this paper, we continue the study of the function $\ex(n, T,\{K_{k+1},M_{s+1}\})$. We determine the exact value of $\ex(n,K_r,\{K_{k+1},M_{s+1}\})$ for $r\ge 3$ and the exact value of $\ex(n,S_r,\{K_{k+1},M_{s+1}\})$ for $n\ge 2(s+1)(r+1)$ and $r\ge 2$.
\end{abstract}

\section{Introduction}
Let $G=(V,E)$ be a graph with vertex set $V=V(G)$ and edge set $E=E(G)\subset\binom{V}{2}$. We may write $G$ instead of $E(G)$. 

Let $T$ be a fixed graph and $\mathcal{H}$ be a set of given graphs. A graph $G$ is called {\em $\mathcal{H}$-free} if $G$ contains no copy of any member in $\mathcal{H}$ as its subgraph.  Write $N(G,T)$ for the number of copies of $T$ in a graph $G$. 
Define the generalized Tur\'an number as 
$$\ex(n, T, \mathcal{H})=\max\{N(G, T) : G \text{ is an $n$-vertex $\mathcal{H}$-free graph}\}.$$ 
We call an $n$-vertex graph $G$ with $N(G, T)$ attaining the maximum an {\em extremal graph} of $\mathcal{H}$.
This function has been systematically studied by Alon and Shikhelman~\cite{2016Many} and has received much attention, for example, in~\cite{G21DMGT,G21DM,G21arX,G22arX2,GERBNER2020169,GP19,GP20,Hei21,2020Few,doi:10.1137/19M1239052}. 
When $T=K_2$,  it is the classical Tur\'an number $\ex(n, \mathcal{H})$.

Let $K_r$ denote a complete graph on $r$ vertices for some integer $r$.
For a set  $U$, write $K[U]$ for a complete graph on vertex set $U$. 
Let  $U_1,U_2,\ldots,U_r$ be disjoint sets and $\mathcal{U}=\{U_1,\ldots,U_r\}$, write $K[\mathcal{U}]=K[U_1,U_2,\ldots,U_r]$ for a complete $r$-partite graph with partition sets $U_1,\ldots,U_r$.
Let $G=(V, E)$ be a graph. For a set $U\subseteq V$, write $G[U]$ for the subgraph induced by $U$. 
For disjoint  sets $U_1,U_2,\cdots,U_r\subseteq V$, write $G[U_1,\cdots, U_r]$ for the induced $r$-partite subgraph of $G$, i.e. $G[U_1,\cdots, U_r]=K[U_1,\cdots, U_r]\cap G$. Let $K_{a_1,a_2,\cdots,a_r}$ denote the complete $r$-partite graph with partition sets of size $a_1,a_2,\cdots,a_r$. In particular, $S_r=K_{1,r}$ is the star with $r$ edges. We call the only vertex incident to all the edges in a star be the center of the star and call every other vertex a leaf of the star.
For graphs $G_1,\cdots, G_r$, let $\sum_{i=1}^rG_i$ be the union of vertex-disjoint copies of $G_1,\cdots, G_r$. 

A Tur\'an graph $T_{k}(n)$ is a complete $k$-partite graph on $n$ vertices whose partition sets have sizes as equal as possible. Let $t_{k}(n)=|T_{k}(n)|=N(T_k(n),K_2)$.

The famous Tur\'an Theorem~\cite{Tr} states that $\ex(n,K_2,K_{k+1})= t_{k}(n)$.
Zykov~\cite{Zykov} gave the  generalized version of Tur\'an Theorem as follows.

\begin{thm}[\cite{Zykov}, see also \cite{GTr}]\label{GTr}
For all $n\ge k\ge r\ge 2$, $$ex(n,K_r,K_{k+1})=N(T_k(n),K_r),$$ and $T_{k}(n)$ is the unique extremal graph.
\end{thm}

Write $\chi(G)$ for the chromatic number of graph $G$.
We say a graph is {\it edge-critical} if there exists some edge whose deletion reduces its chromatic number. Simonovits~\cite{Critical} proved that for any edge-critical graph $H$ with $\chi(H)=k+1\ge 3$, $\ex(n,K_2,H)=t_k(n)$ for sufficiently large $n$, and $T_{k}(n)$ is the unique extremal graph. This result was extended by Ma and Qiu~\cite{MQ} as follows:  For sufficiently large $n$, $\ex(n,K_r,H)=N(T_k(n),K_r)$,  and $T_{k}(n)$ is the unique extremal  graph, where $H$ is  an edge-critical graph with $\chi(H)=k+1>r\ge 2$.

Let $G_k(n,s)=\overline{K_{n-s}}\vee T_{k-1}(s)$, the join of an  empty graph $\overline{K_{n-s}}$ and Tur\'an graph $T_{k-1}(s)$,  i.e. a complete $k$-partite graph on $n$ vertices with one partition set of size $n-s$ and the others having sizes as equal as possible. 
Write $M_k$ for a matching consisting of $k$ edges.
Another fundamental result in graph theory is the
Erd\H{o}s-Gallai Theorem, showing that $$\ex(n, K_2, M_{s+1})=\max\left\{\left|E(G_{s+1}(n, s))\right|, {2s+1\choose 2}\right\}.$$ 
Recently, Alon and Frankl~\cite{AF} studied the function $\ex(n,K_2,\mathcal{H})$ when $\mathcal{H}=\{K_{k+1}, M_{s+1}\}$.

\begin{thm}[\cite{AF}]\label{AF}
For $n\ge 2s+1$ and $k\ge 2$, 
$$\ex(n,K_2,\{K_{k+1},M_{s+1}\})=\max\left\{|T_k(2s+1)|,|G_k(n,s)|\right\}\mbox{.}$$
\end{thm}

In this article, we first extend the result of Alon and Frankl~\cite{AF} as shown in the following.

\begin{thm}\label{main1}
For $n\ge 2s+1$ and $k\ge r\ge 3$, 
$$\ex(n,K_r,\{K_{k+1},M_{s+1}\})=\max\{{N}(T_k(2s+1),K_r),{N}(G_k(n,s),K_r)\}\mbox{.}$$
\end{thm}

We also give a similar result counting the number of stars.

\begin{thm}\label{mainstar}
For $n\ge 2(s+1)(r+1)$ and $r\ge 2$,
$$\ex(n,S_r,\{K_{k+1},M_{s+1}\})={N}(G_k(n,s),S_r)\mbox{.}$$
\end{thm}








\section{Preliminaries}



We need the following fundamental theorem in graph theory.

\begin{thm}[Tutte-Berge Theorem~\cite{Tutte-Berge1}, see also \cite{Tutte-Berge2}]\label{Tutte-Berge}
A graph $G$ is $M_{s+1}$-free if and only if there is a set $B\subset V(G)$ such that  all the components $G_1,\ldots, G_m$ of $G-B$ are odd (i.e. $|V(G_i)|\equiv1\pmod2$ for $i\in[m]$), and 
$$|B|+\sum_{i=1}^{m}\frac{|V(G_i)|-1}{2}\le s\mbox{.}$$
\end{thm}

Let $G$ be a graph and $r$ be a positive integer. For a vertex $v\in V(G)$, define $N_{G}^{(r)}(v)=\{U\in \binom{V}{r}: G[U\cup\{v\}]\cong K_{r+1}\}$ be the {\em $r$-clique neighborhood} of $v$ and $d_{G}^{(r)}(v)=|N_{G}^{(r)}(v)|$ be the {\em $r$-clique-degree} of $v$. For a set $U\subset V(G)$, write $d_G^{(r)}(U)=\sum_{u\in U}d_G^{(r)}(u)$. As usual, write neighborhood $N_{G}(v)$ and degree $d_{G}(v)$ instead of $1$-clique neighborhood $N_{G}^{(1)}(v)$ and $1$-clique-degree $d_{G}^{(1)}(v)$ for short. 

For two non-adjacent vertices $u,v$ in a graph $G$, we  define the {\em switching operation $u\rightarrow v$} as deleting the edges joining $u$ to its neighbors  and adding new edges connecting $u$ to vertices in $N_G(v)$. Let  $G_{u\rightarrow v}$ be the graph
obtained from $G$ by the switching operation $u\rightarrow v$, that is $V(G_{u\rightarrow v})=V(G)$ and
$$E(G_{u\rightarrow v})=\left(E(G)\setminus E(G[\{u\}, N_G(u)])\right)\cup E(G[\{u\}, N_G(v)]).$$
Note that the edges between  $u$ and the common neighbors of $u$ and $v$ remain unchanged by the definition of $G_{u\rightarrow v}$.
For two disjoint independent sets $S$ and $T$ in a graph $G$, if all of the vertices in $S$ (resp. $T$) have the same neighborhood $N_G(S)$ (resp. $N_G(T)$) and $E(G[S, T])=\emptyset$, we then call $\{S,T\}$ an {\it independent pair}.   
For any independent pair $\{S,T\}$, we similarly define $G_{S\rightarrow T}$ to be the graph obtained from $G$ by deleting the edges between $S$  and $N_G(S)$ and adding new edges connecting $S$ and $N_G(T)$. 

For any fixed graph $H$, we say $H$ is {\it switchable} if for any graph $G$ and any independent pair $\{S,T\}$ in $G$, either $G'=G_{S\rightarrow T}$ or $G'=G_{T\rightarrow S}$ has the property that $N(G', H)\ge N(G, H)$.

\begin{prop}\label{PROP:p1}

For $r\ge 2$, $K_r$ and $S_r$ are both switchable.
\end{prop}
\begin{proof}
Let $\{S,T\}$ be an independent pair of $G$.

We firstly prove for $K_r$.
Without loss of generality, suppose $d_{G}^{(r-1)}(T)\ge d_{G}^{(r-1)}(S)$. Let $G'=G_{S\rightarrow T}$.
Then
$${N}(G',K_r)={N}(G,K_r)-d_{G}^{(r-1)}(S)+d_{G}^{(r-1)}(T)\ge N(G,K_r)\mbox{,}$$	
the equality holds if and only if $d_{G}^{(r-1)}(T)=d_{G}^{(r-1)}(S)$.

For $S_r$, note that ${N}(G,S_r)=\sum_{v\in G}\binom{d_G(v)}{r}$. Let $X=N(G_{S\rightarrow T},S_r)-N(G,S_r)$ and $Y=N(G_{T\rightarrow S},S_r)-N(G,S_r)$, then
\begin{eqnarray*}
X&=&|S|\left(\binom{|N_{G}(T)|}{r}-\binom{|N_{G}(S)|}{r}\right)\\
&+&\sum_{u\in N_{G}(S)\backslash N_{G}(T)}\left(\binom{d_{G}(u)-|S|}{r}-\binom{d_{G}(u)}{r}\right)\\
&+&\sum_{w\in N_{G}(T)\backslash N_{G}(S)}\left(\binom{d_{G}(w)+|S|}{r}-\binom{d_{G}(w)}{r}\right),
\end{eqnarray*}
and 
\begin{eqnarray*}
Y&=&|T|\left(\binom{|N_{G}(S)|}{r}-\binom{|N_{G}(T)|}{r}\right)\\
&+&\sum_{u\in N_{G}(S)\backslash N_{G}(T)}\left(\binom{d_{G}(u)+|T|}{r}-\binom{d_{G}(u)}{r}\right)\\
&+&\sum_{w\in N_{G}(T)\backslash N_{G}(S)}\left(\binom{d_{G}(w)-|T|}{r}-\binom{d_{G}(w)}{r}\right).
\end{eqnarray*}
Therefore, 
\begin{eqnarray*}
|T|X+|S|Y&=&\sum_{u\in N_{G}(S)\backslash N_{G}(T)}\left(|S|\binom{d_{G}(u)+|T|}{r}+|T|\binom{d_{G}(u)-|S|}{r}\right)\\
&-&\sum_{u\in N_{G}(S)\backslash N_{G}(T)}(|S|+|T|)\binom{d_{G}(u)}{r}\\
&+&\sum_{w\in N_{G}(T)\backslash N_{G}(S)}\left(|T|\binom{d_{G}(w)+|S|}{r}+|S|\binom{d_{G}(w)-|T|}{r}\right)\\
&-&\sum_{w\in N_{G}(T)\backslash N_{G}(S)}(|S|+|T|)\binom{d_{G}(w)}{r}.
\end{eqnarray*}
Now since $\binom{x}{r}$ is a convex function of $x$, by Jensen's inequation, $|T|X+|S|Y\ge 0$. This means either $X\ge 0$ or $Y\ge 0$, which completes the proof.
\end{proof}

Let $G$ be a graph and $u_1$, $u_2$, $v_1$, $v_2$ be distinct vertices in $G$. Define $G_{(u_1,v_1)\rightarrow(u_2,v_2)}=(G_{u_1\rightarrow u_2})_{v_1\rightarrow v_2}=(G_{v_1\rightarrow v_2})_{u_1\rightarrow u_2}$.
\begin{prop}\label{PROP:p+}
Let $G$ be a graph and $u_1$, $u_2$, $v_1$, $v_2$ be distinct vertices in $G$. Suppose $u_1u_2, v_1v_2, u_1v_2, v_1u_2\notin E(G)$. Then either $G'=G_{(u_1,v_1)\rightarrow(u_2,v_2)}$ or $G'=G_{(u_2,v_2)\rightarrow(u_1,v_1)}$ has $N(G',S_r)\ge N(G,S_r)$ for $r\ge 2$.
\end{prop}
\begin{proof}
For any two vertices $x$ and $y$ in $G$, let $\varepsilon_{xy}=1$ if $xy\in E(G)$ and  $\varepsilon_{xy}=0$ otherwise.
For $i=1,2$, Let $\varepsilon_i=\varepsilon_{u_iv_i}$. Also, for any vertex $v\in V(G)$, let $\varepsilon_v=\varepsilon_{vu_2}+\varepsilon_{vv_2}-\varepsilon_{vu_1}-\varepsilon_{vv_1}$.
Let $X=N(G_{(u_1,v_1)\rightarrow(u_2,v_2)},S_r)-N(G,S_r)$ and $Y=N(G_{(u_2,v_2)\rightarrow(u_1,v_1)},S_r)-N(G,S_r)$.
Similarly as the proof in Proposition~\ref{PROP:p1}, we see that
\begin{eqnarray*}
X&=&2\left(\binom{d_{G}(u_2)+\varepsilon_2}{r}+\binom{d_{G}(v_2)+\varepsilon_2}{r}\right)\\
&-&\left(\binom{d_{G}(u_1)}{r}+\binom{d_{G}(v_1)}{r}+\binom{d_{G}(u_2)}{r}+\binom{d_{G}(v_2)}{r}\right)\\
&+&\sum_{v\neq u_1,u_2,v_1,v_2}\left(\binom{d_{G}(v)+\varepsilon_v}{r}-\binom{d_{G}(v)}{r}\right),
\end{eqnarray*}
and
\begin{eqnarray*}
Y&=&2\left(\binom{d_{G}(u_1)+\varepsilon_1}{r}+\binom{d_{G}(v_1)+\varepsilon_1}{r}\right)\\
&-&\left(\binom{d_{G}(u_1)}{r}+\binom{d_{G}(v_1)}{r}+\binom{d_{G}(u_2)}{r}+\binom{d_{G}(v_2)}{r}\right)\\
&+&\sum_{v\neq u_1,u_2,v_1,v_2}\left(\binom{d_{G}(v)-\varepsilon_v}{r}-\binom{d_{G}(v)}{r}\right).
\end{eqnarray*}
Therefore,
\begin{eqnarray*}
X+Y&=&2\left(\binom{d_{G}(u_1)+\varepsilon_1}{r}+\binom{d_{G}(v_1)+\varepsilon_1}{r}+\binom{d_{G}(u_2)+\varepsilon_2}{r}+\binom{d_{G}(v_2)+\varepsilon_2}{r}\right)\\
&-&2\left(\binom{d_{G}(u_1)}{r}+\binom{d_{G}(v_1)}{r}+\binom{d_{G}(u_2)}{r}+\binom{d_{G}(v_2)}{r}\right)\\
&+&\sum_{v\neq u_1,u_2,v_1,v_2}\left(\binom{d_{G}(v)-\varepsilon_v}{r}+\binom{d_{G}(v)+\varepsilon_v}{r}-2\binom{d_{G}(v)}{r}\right).
\end{eqnarray*}
Note that $\varepsilon_1,\varepsilon_2\ge 0$. Now since $\binom{x}{r}$ is a convex function of $x$, by Jensen's inequation, $X+Y\ge 0$. This means either $X\ge 0$ or $Y\ge 0$, which completes the proof.
\end{proof}

We give the following observation about combination numbers without proof.
\begin{obs}\label{bin}
For integers $n> m\ge r\ge 2$, 
$$1-\frac{r(n-m)}{n-r+1}<\left(1-\frac{n-m}{n-r+1}\right)^{r}=\left(\frac{m-r+1}{n-r+1}\right)^r<\frac{\binom{m}{r}}{\binom{n}{r}}<\left(\frac{m}{n}\right)^r.$$ 
\end{obs}

\begin{prop}\label{calc}
Let $r\ge 2$, $s\ge 1$, $k\ge 2$, $0\le b\le s$ and $n\ge 2(s+1)(r+1)$ be integers. Let $h_{n,k}(x_1,\cdots, x_{k-1})=\sum_{i=1}^{k-1}x_i\binom{n-x_i}{r}$.\\
(1) Let
$$g_{n,k,r,b,y}(x_1,\cdots, x_k):=h_{n,k}(x_1,\cdots, x_{k-1})+x_k\binom{y-x_k}{r}$$
be a function on $\{(x_1,\cdots, x_k)\in \mathbb{Z}^k: x_1+\cdots+x_k=b\mbox{, }x_i\ge 0 \mbox{ for any } i\in[k]\}$, where $s\le y\le 2s$ is an integer. Then $g_{n,k,r,b,y}(x_1,\cdots, x_k)$ reaches its maximum when $x_k=0$ and $|x_i-x_j|\le 1$ for any $i,j\in[k-1]$.\\
(2) Let $x=\sum_{i=1}^{k-1}x_i$ and
$$g_{n,k,r,s}(x_1,\cdots, x_{k-1}):=h_{n,k}(x_1,\cdots,x_{k-1})+(2s+1-2x)\binom{2s-x}{r}+(n-1-2s+x)\binom{x}{r}$$
be a function on $\{(x_1,\cdots, x_{k-1})\in \mathbb{Z}^{k-1}: x\le s\mbox{, }x_i\ge 0 \mbox{ for any } i\in[k-1]\}$. Then the function $g_{n,k,r,s}(x_1,\cdots, x_{k-1})$ reaches its maximum when $x=s$ and $|x_i-x_j|\le 1$ for any $i,j\in[k-1]$.\\
(3) $N(G_k(n,s),S_r)$ is an increasing function of $s$.
\end{prop}
\begin{proof}
before the proof of (1), (2) and (3), we firstly prove the following claim.
\begin{claim}\label{fturan}
On $\{(x_1,\cdots, x_{k-1})\in \mathbb{Z}^{k-1}: x_1+\cdots+x_{k-1}=b\mbox{, }x_i\ge 0 \mbox{ for any } i\in[k-1]\}$, $h_{n,k}(x_1,\cdots, x_{k-1})$ reaches its maximum when $|x_i-x_j|\le 1$ for any $i,j\in[k-1]$.
\end{claim}
\begin{proof}
If $k=2$, then we are done.
For $k\ge 3$, without loss of generality, suppose on the contrary that $h_{n,k}$ reaches its maximum on $(b_1,\cdots, b_{k-1})$ while $b_1-b_2\ge 2$. Let $X=h_{n,k}(b_1-1, b_2+1, b_3,\cdots, b_{k-1})-h_{n,k}(b_1,b_2,b_3,\cdots,b_{k-1})$, then
\begin{eqnarray*}
X&=&(b_1-1)\binom{n-b_1+1}{r}+(b_2+1)\binom{n-b_2-1}{r}-b_1\binom{n-b_1}{r}-b_2\binom{n-b_2}{r}\\
&=&\left(\binom{n-b_2-1}{r}-\binom{n-b_1+1}{r}\right)+\left(b_1\binom{n-b_1}{r-1}-b_2\binom{n-b_2-1}{r-1}\right)\\
&=&\left(\frac{n-b_2-r}{r}\binom{n-b_2-1}{r-1}-\frac{n-b_1+1}{r}\binom{n-b_1}{r-1}\right)+\left(b_1\binom{n-b_1}{r-1}-b_2\binom{n-b_2-1}{r-1}\right)\\
&=&\frac{n-(r+1)(b_2+1)+1}{r}\binom{n-b_2-1}{r-1}-\frac{n-(r+1)b_1+1}{r}\binom{n-b_1}{r-1}\\
&>& 0\mbox{,}
\end{eqnarray*}
where the last inequality holds because $n-(r+1)(b_2+1)+1>n-(r+1)b_1+1>0$ and $n-b_2-1>n-b_1\ge r-1$. This is a contradiction by the maximality of $h_{n,k}(b_1\cdots,b_{k-1})$.
\end{proof}
By the proof of the Claim~\ref{fturan}, it is easy to see that (3) is correct since $N(G_k(n,s),S_r)=h_{n,k+1}(x_1,\cdots, x_{k-1}, n-s)$ with $\sum_{i=1}^{k-1}x_i=s$, $|x_i-x_j|\le 1$ for any $i,j\in[k-1]$ and so $|(n-s)-x_i|>1$ for $i\in[k-1]$. It remains to prove (1) and (2).

(1) By Claim~\ref{fturan}, it remains to prove $x_k=0$ when $g_{n,k,r,b,y}$ reaches its maximum. Suppose on the contrary that $g_{n,k,r,b,y}$ reaches its maximum on $(b_1, \cdots, b_k)$ while $b_k\ge 1$. Let $X=g_{n,k,r,b,y}(b_1+b_k,b_2, \cdots, b_{k-1},0)-g_{n,k,r,b,y}(b_1,b_2,\cdots, b_{k-1}, b_{k})$, then
\begin{eqnarray*}
X&=& (b_1+b_k)\binom{n-b_1-b_k}{r}-\left(b_1\binom{n-b_1}{r}+b_k\binom{y-b_k}{r}\right)\\
&>&(b_1+b_k)\left(1-\frac{rb_k}{n-b_1-r+1}\right)\binom{n-b_1}{r}-\left(b_1\binom{n-b_1}{r}+b_k\binom{y-b_k}{r}\right)\\
&=&b_k\left(\left(1-\frac{r(b_1+b_k)}{n-b_1-r+1}\right)\binom{n-b_1}{r}-\binom{y-b_k}{r}\right)\\
&>&b_k\binom{y-b_k}{r}\left(\left(1-\frac{r(b_1+b_k)}{n-b_1-r+1}\right)\left(\frac{n-b_1}{y-b_k}\right)^r-1\right)\\
&\ge&b_k\binom{y-b_k}{r}\left(\left(1-\frac{sr}{2sr}\right)\left(\frac{2sr}{2s}\right)^r-1\right)\\
&\ge&0\mbox{,}
\end{eqnarray*}
where the first two inequality holds by Observation~\ref{bin} and the last holds since $b_1+b_k\le s$, $n-b_1-r+1> 2sr$, $n-b_1>2sr$ and $y-b_k<2s$. This is a contradiction by the maximality of $g_{n,k,r,b,y}(b_1\cdots,b_{k})$.

(2) Similarly, we only need to prove $x=s$ when $g_{n,k,r,s}$ reaches its maximum. Suppose on the contrary that $g_{n,k,r,s}$ reaches its maximum on $(b_1, \cdots, b_{k-1})$ while $x=b_1+\cdots+b_{k-1}\le s-1$. Note that $p(x)=(n-1-2s+x)\binom{x}{r}$ is an increasing function on $x\in[s]$. Let $X=g_{n,k,r,s}(b_1+(s-x),b_2, \cdots, b_{k-1})-g_{n,k,r,s}(b_1,b_2,\cdots, b_{k-1})$, then
\begin{eqnarray*}
X&=&(b_1+(s-x))\binom{n-b_1-(s-x)}{r}-b_1\binom{n-b_1}{r}-(2s+1-2x)\binom{2s-x}{r}\\
&+&\binom{s}{r}+p(s)-p(x)\\
&>&(b_1+(s-x))\binom{n-b_1-(s-x)}{r}-b_1\binom{n-b_1}{r}-(2s+1-2x)\binom{2s-x}{r}\\
&>&(b_1+(s-x))\left(1-\frac{r(s-x)}{n-b_1-r+1}\right)-b_1\binom{n-b_1}{r}-(2s+1-2x)\binom{2s-x}{r}\\
&=&(s-x)\left(1-\frac{r(b_1+(s-x))}{n-b_1-r+1}\right)\binom{n-b_1}{r}-(2s+1-2x)\binom{2s-x}{r}\\
&>&\left((s-x)\left(1-\frac{r(b_1+(s-x))}{n-b_1-r+1}\right)\left(\frac{n-b_1}{2s-x}\right)^r-(2s+1-2x)\right)\binom{2s-x}{r}\\
&>&\left((s-x)\left(\left(1-\frac{sr}{2sr}\right)\left(\frac{2s(r+\frac{1}{2})}{2s}\right)^r-2\right)-1\right)\binom{2s-x}{r}\\
&>&0\mbox{,}
\end{eqnarray*}
where the second inequality and the third inequality hold by Observation~\ref{bin}; the forth inequality holds since $r(b_1+(s-x))\le sr$, $n-b_1-r+1>2sr$, $n-b_1>2s(r+\frac{1}{2})$ and $2s-x\le 2s$. This is a contradiction by the maximality of $g_{n,k,r,s}(b_1\cdots,b_{k-1})$.
\end{proof}

Let $\Delta_{t,k}^r={N}(T_{k}(t),K_{r})$ for some positive integers $t, k, r$. 

\begin{obs}\label{calculation}
	(1) For positive integers $t, k, r\ge 2$, 
	$$\Delta_{t+1,k}^r-\Delta_{t,k}^r=\Delta_{t-\lfloor\frac{t}{k}\rfloor,k-1}^{r-1} \mbox{ and } \Delta_{t,k}^r=\Delta_{t-\lfloor\frac{t}{k}\rfloor,k-1}^{r}+\left\lfloor\frac{t}{k}\right\rfloor\Delta_{t-\left\lfloor\frac{t}{k}\right\rfloor,k-1}^{r-1}\mbox{.}$$
	(2) For $n\ge 2s+1$ and $s\ge t, k, r\ge 3$,  define 
	$$g_{n,k,r}(t):=(n-t)\Delta_{t,k-1}^{r-1}+\Delta_{t,k-1}^r\mbox{.}$$
Then $g_{n,k,r}(t)$ is a strictly increasing function of $t$. In particular, $g_{n,k,r}(s)={N}(G_k(n,s),K_r)$.
\end{obs}
\begin{proof}
(1) can be checked directly by the definitions of $\Delta_{t,k}^r$ and the Tur\'an graph.
	
(2)	By (1), for $t\le s-1$,
	\begin{eqnarray*}
		g_{n,k,r}(t+1)-g_{n,k,r}(t)&=&(n-t-1)\Delta_{t-\lfloor\frac{t}{k-1}\rfloor,k-2}^{r-2}-\Delta_{t,k-1}^{r-1}+\Delta_{t-\lfloor\frac{t}{k-1}\rfloor,k-2}^{r-1}\\
		&=&(n-t-1)\Delta_{t-\lfloor\frac{t}{k-1}\rfloor,k-2}^{r-2}-\left\lfloor\frac{t}{k-1}\right\rfloor\Delta_{t-\lfloor\frac{t}{k-1}\rfloor,k-2}^{r-2}\\
		&=&\left(n-1-t-\left\lfloor\frac{t}{k-1}\right\rfloor\right)\Delta_{t-\lfloor\frac{t}{k-1}\rfloor,k-2}^{r-2}>0\mbox{.}
	\end{eqnarray*}
	This completes the proof.
\end{proof}

\section{Proof of Theorem~\ref{main1} and Theorem~\ref{mainstar}}

Now we are ready to give the proofs of Theorem~\ref{main1} and Theorem~\ref{mainstar}. 

	Let $G$ be an extremal graph of  $\{K_{k+1},M_{s+1}\}$ (with the maximum number of $K_r$ (resp. $S_r$) in it) on $n\ge 2s+1$ vertices.
By Theorem~\ref{Tutte-Berge}, there is a vertex set $B\subset V(G)$ such that  $G-B$ consists of odd components $G_1,\ldots, G_m$, and 
$$|B|+\sum_{i=1}^{m}\frac{|V(G_i)|-1}{2}\le s\mbox{.}$$
Let $A_i=V(G_i)$ and $|A_i|=a_i$ for $i\in[m]$. Denote $A=\cup_{i=1}^mA_i$. Let $I_G(A)=\{ i\in[m] : a_i=1\}$. We may choose $G$ maximizing $|I_G(A)|$ (assumption (*)). Let $|B|=b$. 

Define two vertices $u$ and $v$ in  $B$ are equivalent if and only if $N_G(u)=N_G(v)$. Clearly, it is an equivalent relation. Therefore, the vertices of $B$ can be partitioned into equivalent classes according to the equivalent relation defined above.  We may choose $G$ (among graphs $G$ satisfying assumption (*)) with the minimum number of equivalent classes of $B$ (assumption (**)).  
Note that each equivalent class of $B$ is an independent set of $G$ by the definition of the equivalent relation. 
 We firstly claim that every two non-adjacent vertices of $B$ have the same neighborhood (a clique version of Lemma 2.1 of~\cite{AF}),  which is also a simple consequence of the Zykov symmetrization method introduced in~\cite{Zykov}, for completeness we include the proof.

\begin{lem}\label{B}
Every two non-adjacent vertices of $B$ have the same neighborhood.
\end{lem}
\begin{proof}
Suppose there are two non-adjacent vertices $u, w\in B$ with $N_{G}(u)\neq N_{G}(w)$. Then $u$ and $w$ must be in distinct equivalent classes $U$ and $W$ by the definition of the equivalence. Since $uw\notin E(G)$, we have $E(G[U, W])=\emptyset$. Thus, $\{U,W\}$ is an independent pair. Without loss of generality, let $G'=G_{U\rightarrow W}$. By Proposition~\ref{PROP:p1}, we can suppose $N(G', K_r)\ge N(G, K_r)$ (resp. $N(G', S_r)\ge N(G, S_r)$). 
Now we show that $G'$ is $\{K_{k+1}, M_{s+1}\}$-free too. Clearly,  $G'-B$ still consists of odd components $G_1,\ldots, G_m$. 
Hence $G'$ is $M_{s+1}$-free by  Theorem~\ref{Tutte-Berge}. 
If $G'$ contains a copy $T$ of $K_{k+1}$, we must have a vertex $u\in V(T)\cap U$. Choose a vertex $w\in W$. Since $N_{G'}(u)=N_{G'}(w)=N_G(w)$,  $(V(T)\setminus\{u\})\cup\{w\}$ induces a copy of $K_{k+1}$ in $G$, a contradiction. Hence, $G_{U\rightarrow W}$ is $\{K_{k+1}, M_{s+1}\}$-free. 
By the extremality of $G$, we have $N(G', K_r)= N(G, K_r)$ (resp. $N(G', S_r)= N(G, S_r)$). 
 But the number of equivalent classes of $G'$ ($U$ and $W$  merge into one class in $G'$) is less than the one in $G$, a contradiction to assumption (**).
\end{proof}

In fact, with similar proofs ,we also have the following two lemmas.
\begin{lem}\label{Ai}
For any fixed $i\in [m]$, every two non-adjacent vertices of $A_i$ have the same neighborhood.
\end{lem}
\begin{lem}\label{Ai1}
For  $i,j\in [m]$ with $a_i=a_j=1$, the only vertex $v_i\in A_i$ and the only vertex $v_j\in A_j$ have the same neighborhood.
\end{lem}


Here we firstly give the proof of Theorem~\ref{main1}.

\begin{proof}[Proof of Theorem~\ref{main1}:]
Let $G$ be the extremal graph mentioned above  with the maximum number of $K_r$ .
By Lemma~\ref{B} and $G$ being $K_{k+1}$-free, $G[B]$ is a complete $\ell$-partite graph with $\ell\le k$. Let its partition sets be $B_1, \ldots, B_\ell$ and let $B_{\ell+1}=\cdots=B_k=\emptyset$ if $\ell<k$. Let $b_i=|B_i|$ for $i\in[k]$. Without loss of generality, assume $b_1\ge b_2\ge\ldots\ge b_k\ge 0$. Write $\mathcal{B}=\{B_1,\ldots,B_{k-1}\}$. Let $\Delta_{\mathcal{B}}^{r-1}={N}(K[\mathcal{B}],K_{r-1})$.
 Since $\sum_{i=1}^{k-1}b_i=b-b_k$, by Theorem~\ref{GTr}, $\Delta_{\mathcal{B}}^{r-1}\le\Delta_{b-b_k,k-1}^{r-1}$.



Recall that $A=\cup_{i=1}^m A_i$. For those isolated vertices in $G[A]$, we have the following claim.
\begin{claim}\label{independent}
For $v\in A$ with $d_{G[A]}(v)=0$, we have $d_{G}^{(r-1)}(v)\le \Delta_{\mathcal{B}}^{r-1}\le\Delta_{b-b_k,k-1}^{r-1}$.
\end{claim}
\begin{proof}
Let $T$ be the vertex set of a copy of $K_r$ covering $v$. Then $T\cap A=\{v\}$ and $|T\cap B|=r-1$ because $d_{G[A]}(v)=0$. Therefore, $G[T\cap B]\cong K_{r-1}$.  Note that $N_G(v)\subset B$. We have $d_{G}^{(r-1)}(v)\le N(G[N_{G}(v)],K_{r-1})$. 

Let $I(v)=\{i\in[k] : N_G(v)\cap B_i\not=\emptyset\}$.  Apparently, $|I(v)|\le k-1$. Otherwise, $I(v)=[k]$. Note that $G[B]=K[B_1, \ldots, B_k]$ in this case. Thus $G[B\cup\{v\}]=K[\{v\}, B_1,\ldots, B_k]$ contains a copy of $K_{k+1}$, a contradiction.
Now let $\mathcal{B}'=\cup_{i\in I(v)}\{B_i\}\subseteq \mathcal{B}$. Then 
$$d_{G}^{(r-1)}(v)\le N(G[N_{G}(v)],K_{r-1})\le N(K[\mathcal{B}'], K_{r-1})\le N(K[\mathcal{B}],K_{r-1})=\Delta_{\mathcal{B}}^{r-1}.$$
\end{proof}

When $A$ is an independent set of $G$, we have the following claim. 
\begin{claim}\label{A=0}
If $A$ is an independent set of $G$, then $G\cong G_{k}(n,s)$.
\end{claim}
\begin{proof}
Let $\mathcal{B}_0=\{B_1,\cdots, B_k\}$. Recall that $\mathcal{B}=\{B_1,\ldots,B_{k-1}\}$.
Then $$N(K[\mathcal{B}_0],K_r)=\Delta_{\mathcal{B}_0}^r=\Delta_{\mathcal{B}}^{r-1}|B_k|+\Delta_{\mathcal{B}}^r.$$
Note that $b+\sum_{i=1}^m\frac{a_i-1}2\le s$. Hence when $A=\cup_{i=1}^m A_i$ is an independent set, we have $a_1=\ldots=a_m=1$ and $b\le s$. We can also suppose $b-b_k\ge 3$ since the small cases are easy to check.
By Claim~\ref{independent}, 
\begin{eqnarray*}
	N(G,K_r)&\le&\Delta_{\mathcal{B}_0}^r+\sum_{v\in A}d^{r-1}_G(v)\\
	        &\le  &\Delta_{\mathcal{B}_0}^r+(n-b)\Delta_{\mathcal{B}}^{r-1}\\
	        &=&(n-b+|B_k|)\Delta_{\mathcal{B}}^{r-1}+\Delta_{\mathcal{B}}^r\\
	        &\le& [n-(b-b_k)]\Delta_{b-b_k,k-1}^{r-1}+\Delta_{b-b_k,k-1}^{r}\\
	        &=& g_{n,k,r}(b-b_k)\mbox{.}
\end{eqnarray*}
By Observation~\ref{calculation} (2), we have  
$$N(G,K_r)\le g_{n,k,r}(b-b_k)\le g_{n,k,r}(s)=N(G_k(n,s), K_r).$$
 When the equality holds, we must have $b_k=0$, $b=s$, $G[B]\cong T_{k-1}(s)$ by Theorem~\ref{GTr}, and $G[B,A]=K[B,A]$. This implies that $G\cong G_k(n,s)$.
\end{proof}

\begin{claim}\label{A}
$a_2=a_3=\cdots=a_m=1$.
\end{claim}
\begin{proof}
If $A$ is an independent set of $G$, then we are done. Now suppose $|G[A]|>0$. Then $b<s$. 
Let $v_0$ be a vertex in $A$ with $d_{G}^{(r-1)}(v_0)=\max\limits_{v\in A}d_{G}^{(r-1)}(v)$. Without loss of generality, suppose $v_0\in A_1$. 

If $d_{G[A]}(v_0)=0$, let $G'$ be the resulting graph by applying the switching operations $u\rightarrow v_0 $ for all vertices $u\in A\setminus\{v_0\}$ one by one. Then we have $|G'[A]|=0$. By Proposition~\ref{PROP:p1} (and its proof), $N(G',K_r)\ge N(G,K_r)$. With the same discussion as in the proof of Lemma~\ref{B}, we have that $G'$ is still $\{K_{k+1},M_{s+1}\}$-free. But $|I_{G'}(A)|=m>|I_G(A)|$, a contradiction to the assumption (*).

If $d_{G[A]}(v_0)>0$, then $a_1\ge 3$. If $a_2=\ldots=a_m=1$, we are done. Now, without loss of generality, assume $a_2\ge 3$. Since $G[A_2]$ is connected, we can pick two vertices, say $u_1,u_2$ in $A$ such that $G[A_2\backslash\{u_1,u_2\}]$ is still connected ($u_1, u_2$ exist, for example, we can take two leaves of a spanning tree of $G[A_2]$). 
Let $G_1$ be the resulting graph by applying the switching operations $u_1\rightarrow v_0$ and $u_2\rightarrow v_0$ one by one. 
With similar discussion as in the above case, we have $N(G_1,K_r)\ge N(G,K_r)$ and $G_1$ is $\{K_{k+1},M_{s+1}\}$-free. 
Continue the process after $t=\frac{a_2-1}{2}$ steps, we obtain a graph $G_t$ with $N(G_t,K_r)\ge N(G,K_r)$ and $G_t$ is $\{K_{k+1},M_{s+1}\}$-free. But $|I_{G_t}(A)|=|I_G(A)|+1$, a contradiction to the assumption (*).
\end{proof}

Now by Claim~\ref{A} and Theorem~\ref{Tutte-Berge}, we have $s':=b+\frac{a_1-1}{2}\le s$. Thus, $a_1= 2s'-2b+1$ and then $|B\cup A_1|=a_1+b=2s'-b+1$. Note that $0\le b\le s'$. By Claim~\ref{independent}, for vertex $v\not\in B\cup A_1$, $d_{G}^{(r-1)}(v)\le \Delta_{b-b_k,k-1}^{r-1}\le \Delta_{b,k-1}^{r-1}$. Also, since $G[B\cup A_1]$ is $K_{k+1}$-free, by Theorem~\ref{GTr}, $N(G[B\cup A_1],K_r)\le\Delta_{2s'-b+1,k}^{r}$. Therefore,
$$N(G,K_r)\le \Delta_{2s'-b+1,k}^{r}+(n-2s'+b-1)\Delta_{b,k-1}^{r-1}\mbox{.}$$
Define $f_{n,k,r,s'}(b):=\Delta_{2s'-b+1,k}^{r}+(n-2s'+b-1)\Delta_{b,k-1}^{r-1}$.
If $b=0$, then $f_{n,k,r,s'}(0)=N(T_{k}(2s'+1),K_{r})\le N(T_{k}(2s+1),K_{r})$, and the proof is done. If $b=s'$, then $a_1=1$ and thus $A$ is an independent set of $G$, the proof is done by Claim~\ref{A=0}. In particular, $f_{n,k,r,s'}(s')\le N(G_k(n,s'),K_r)\le N(G_k(n,s),K_r)$.

By Observation~\ref{calculation} (1), for $0\le b\le s'-1$,
$$f_{n,k,r,s'}(b+1)-f_{n,k,r,s'}(b)=-\Delta_{(2s'-b)-\lfloor\frac{2s'-b}{k}\rfloor,k-1}^{r-1}+(n-2s'+b)\Delta_{b-\lfloor\frac{b}{k-1}\rfloor,k-2}^{r-2}+\Delta_{b,k-1}^{r-1}\mbox{.}$$
For fixed $k$ and $r$, $\Delta_{t,k}^r$ is an increasing function of $t$, and  $t-\lfloor\frac{t}{k}\rfloor$ is a non-decreasing function  of $t$.  It is easy to check that $g(b)=f_{n,k,r,s}(b+1)-f_{n,k,r,s'}(b)$ is an increasing function on  $0\le b\le s'-1$. This implies that $f_{n,k,r,s'}(b)$ is convex on $[0, s'-1]$. 
Therefore, 
$$f_{n,k,r,s'}(b)\le \max\{f_{n,k,r,s'}(0),f_{n,k,r,s'}(s')\}\le\max\{N(T_k(2s+1),K_r),N(G_k(n,s),K_r)\}\mbox{.}$$
\end{proof}

Now we prove for Theorem~\ref{mainstar}.
\begin{proof}[Proof of Theorem~\ref{mainstar}:]
Let $G$ be the extremal graph mentioned above  with the maximum number of $S_r$.
\begin{claim}~\label{As}
$a_2=a_3=\cdots=a_m=1$.
\end{claim}
\begin{proof}
For any $i\in [m]$, a pair of vertices $\{u,v\}\subset A_i$ is said to be $2$-switchable if $G[A_i\backslash\{u,v\}]$ is still connected.
For any distinct $i, j\in [m]$ and any two ordered pairs $\{u_1, v_1\}\subset A_i$ and $\{u_2, v_2\}\subset A_j$, we call $(u_1,v_1)\rightarrow (u_2,v_2)$ a $2$-switching if $\{u_1, v_1\}$ is  $2$-switchable. we call it a reversible $2$-switching if both  $\{u_1, v_1\}$ and  $\{u_2, v_2\}$ are $2$-swithable.
A $2$-switching $(u_1,v_1)\rightarrow (u_2,v_2)$ is said to be good if $N(G_{(u_1,v_1)\rightarrow (u_2,v_2)},S_r)>N(G,S_r)$; it is said to be bad if $N(G_{(u_1,v_1)\rightarrow (u_2,v_2)},S_r)<N(G,S_r)$; otherwise, $N(G_{(u_1,v_1)\rightarrow (u_2,v_2)},S_r)=N(G,S_r)$ and we say it is a fair  $2$-switching. Clearly, there is no good $2$-switching in $G$.

Since there is no  good $2$-switching, any $2$-switching $(u_1,v_1)\rightarrow (u_2,v_2)$ is bad or fair. If it is bad and reversible, then by proposition~\ref{PROP:p+},  $(u_2,v_2)\rightarrow (u_1,v_1)$ must be a good $2$-switching, a contradiction. In other words, a reversible $2$-switching cannot be bad, so it must be fair.

By Lemma~\ref{Ai}, for any $i\in[m]$, $G[A_i]$ is a complete multipartite graph. One can check that there is at most one pair within $A_i$ that is not $2$-switchable (only if $a_i\ge 5$ and all the vertices in $A_i$ other than this pair form a partition set of $G[A_i]$). Hence we can always pick some $2$-switchable pair in $A_i$ if $a_i\ge 3$.

Now Suppose on the contrary that $a_1\ge a_2\ge 3$. suppose $\{u_0,v_0\}\subset A_1$ is $2$-switchable. For any $2$-switchable pair $\{u,v\}\in A_i$ for some $i\ge 2$, $\{u,v\}\rightarrow\{u_0,v_0\}$ is reversible, so it must be fair. Thus, $G_{\{u,v\}\rightarrow\{u_0,v_0\}}$ is also an extremal graph. By repeatedly doing $2$-switching $\{u,v\}\rightarrow\{u_0,v_0\}$ for $2$-switchable pair $\{u,v\}\in A_i$ with $i\ge 2$ (note that during these operations, $\{u_0,v_0\}$ keeps being $2$-switchable), we end up with a graph $G'$ which is also an extremal graph. Moreover, there is no $2$-switchable pair in $A_i$ for any $i\ge 2$, which implies that $a_i=1$ for $i\ge 2$. So $I_{G'}(A)=m-1>I_{G}(A)$, a contradiction.
\end{proof}
By Lemma~\ref{B} and $G$ is $K_{k+1}$-free, $G[B]$ is a complete $k$-partite graph (some of its partition set may be empty). Suppose these partition set be $B_1$, $B_2$, $\cdots$, $B_k$ and let $b_i=|B_i|$ for $i\in[k]$. If there exists some $i\in[k]$ with $b_i=0$, then let $B_k=\emptyset$ without loss of generality. Otherwise, by Lemma~\ref{Ai1}, Claim~\ref{As} and since $n\ge 2(s+1)(r+1)$, it holds $A\backslash A_1$ is not empty and all its vertices share the same neighborhood. Clearly, at least one of $B_1$, $\cdots$, $B_k$ is not adjacent to $A\backslash A_1$, or we would get a copy of $K_{k+1}$ in $G$. In this case, we let $B_k$ be a partition set not adjacent to $A\backslash A_1$. 
Now define
$$f'_{n,k,r,s}(b_1,\cdots,b_k,a_1)=g_{n,k,r,b,a_1+b}(b_1,\cdots,b_k)+a_1\binom{a_1+b-1}{r}+(n-a_1-b)\binom{b}{r}.$$
Note that: the leaves of a star whose center is in $B_i$ ($i\neq k$) must be in $V(G)\backslash B_i$; the leaves of a star whose center is in $B_k$ must be in $A_1\cup(B\backslash B_k)$; the leaves of a star whose center is in  $A_1$ must be in $A_1\cup B$; the leaves of a star whose center is in  $A\backslash A_1$ must be in $B$. Therefore, we have $N(G,S_r)\le f'_{n,k,r,s}(b_1,\cdots,b_k,a_1)$.
Now by Proposition~\ref{calc} (1), 
$$f'_{n,k,r,s}(b_1, \cdots, b_{k-1}, b_k, a_1)<f'_{n,k,r,s}(b_1',\cdots, b_{k-1}',0,a_1),$$
where $\sum_{i=1}^{k-1}b_i'=b=\sum_{i=1}^{k}b_i$ and $|b_i'-b_j'|\le 1$ for any $i,j\in[k-1]$.
By Proposition~\ref{calc} (2) and the fact that $a_1=2s'+1-2b$ for some $s'\le s$, 
$$f'_{n,k,r,s'}(b_1',\cdots, b_{k-1}',0,a_1)<f'_{n,k,r,s'}(b_1'',\cdots, b_{k-1}'',0,1),$$
where $\sum_{i=1}^{k-1}b_i''=s'$ and $|b_i''-b_j''|\le 1$ for any $i,j\in[k-1]$.
Therefore, $N(G,S_r)\le f'_{n,k,r,s'}(b_1'',\cdots, b_{k-1}'',0,1)$. 
One can check that $f'_{n,k,r,s'}(b_1'',\cdots, b_{k-1}'',0,1)=N(G_k(n,s'),S_r)\le N(G_k(n,s),S_r)$, which means we are done.
\end{proof}

\section{Concluding Remarks}
In this paper, we study on the function $\ex(n,T,\{K_{k+1},M_{s+1}\})$ and determine the exact values when $T= K_r$ and $T=S_r$. However, for $T=S_r$, our result holds only when $n\ge 2(s+1)(r+1)$. 
One may conjecture that $\ex(n,S_r,\{K_{k+1},M_{s+1}\})=\max\{{N}(T_k(2s+1),S_r),{N}(G_k(n,s),S_r)\}$ for general $n\ge 2s+1$. Unluckily, this is not true. In fact, when $r$ is somehow larger than $k$ (for example $r>20k$), one may find some conterexamples. Hence, we give the following conjecture instead.
\begin{conj}
There exists some fixed positive number $\alpha$, for $n\ge 2s+1$ and $2\le r\le \alpha k$, 
$$\ex(n,S_r,\{K_{k+1},M_{s+1}\})=\max\{{N}(T_k(2s+1),S_r),{N}(G_k(n,s),S_r)\}\mbox{.}$$
\end{conj}

\noindent{\bf Acknowledgment:} {The work was supported by National Natural Science Foundation of China (No.12401455) to Yue Ma;  National Natural Science Foundation of China (No.12471336, 12071453) and Innovation Program for Quantum Science and Technology (2021ZD0302902) to Xinmin Hou.}


\end{document}